\documentclass{article}
\usepackage{mathrsfs}
\usepackage{bbm}
\usepackage{amssymb,amsmath,amsthm,cases}
\usepackage[noadjust]{cite}
\allowdisplaybreaks
\begin{document}
\title {Orthogonalization in Clifford Hilbert modules and applications}
\author{Jinxun Wang\thanks{\scriptsize School of Mathematics and Statistics, Guangdong University of Foreign Studies,
Guangzhou 510006, China. E-mail: wjx@gdufs.edu.cn}, Tao Qian\thanks{\scriptsize Macau Institute of Systems Engineering,
Macau University of Science and Technology, Macau, China. E-mail:
tqian@must.edu.mo}}
\date{}
\maketitle
\newtheorem{defi}{Definition}[section]
\newtheorem{theo}{Theorem}[section]
\newtheorem{pro}{Proposition}[section]
\newtheorem{lem}{Lemma}[section]
\newtheorem{coro}{Corollary}[section]
\newtheorem{exam}{Example}[section]
\newtheorem{rem}{Remark}[section]
\def\D{\bf D}

\noindent\textbf{Abstract:}
We prove that the Gram--Schmidt orthogonalization process can be carried out in Hilbert modules over Clifford algebras, in spite of the un-invertibility and the un-commutativity of general Clifford numbers. Then we give two crucial applications of the orthogonalization method. One is to give a constructive proof of existence of an orthonormal basis of the inner spherical monogenics of order $k$ for each $k\in\mathbb{N}.$ The second is to formulate the Clifford Takenaka--Malmquist systems, or in other words, the Clifford rational orthogonal systems, as well as define Clifford Blaschke product functions, in both the unit ball and the half space contexts. The Clifford TM systems then are further used to establish an adaptive rational approximation theory for $L^2$ functions on the sphere and in $\mathbb{R}^m.$ 


\section{Introduction}
Due to importance of orthonormal bases in both theoretical analysis and real life applications, for a system of functions, $\mathscr{F},$ in a Hilbert space, the questions of existence, and explicit composition if existing, of an orthogonal system equivalent to $\mathscr{F}$ naturally arise. If the functions in $\mathscr{F}$ are complex-valued, existence of an orthonormal basis of $\mathscr{F}$ is guaranteed by the Gram--Schmidt (GS) orthogonalization process, but in the case of Clifford number-valued functions it is not obvious since Clifford algebras are non-commutative and Clifford numbers are usually un-invertible. In Clifford analysis there is an example: How to find an orthonormal basis for the Fueter polynomials of degree $k$ (i.e., the inner spherical monogenics of order $k$)? This problem appears because there are more than one Fueter polynomials of degree $k$ and they are not mutually orthogonal. In \cite{DSS} the existence was proved by induction on dimensions, but no explicit forms were given. In \cite{BGLS} an explicit form was constructed in three dimensions using the Gelfand--Tsetlin bases, but the construction is too complicated for higher dimensions. The mentioned construction is also applicable to the Hermitean Clifford analysis and to some other systems as well (\cite{BDLS,DLS}).

In this paper we show that the GS orthogonalization process can be applied to general Clifford module Hilbert spaces. This is through proving that the orthogonal projection of a function onto the subspace spanned by some other functions exists. We present here a direct construction of an orthonormal basis for a system of Clifford number-valued functions. When we consider some commonly familiar functions like Fueter polynomials and parameterized Szeg\"{o} kernel functions, the construction gives rise to concrete expressions involving inverse of the Clifford-valued inner product. As applications of the obtained fundamental orthogonalization result, we give a constructive proof of existence of an orthonormal basis of the inner spherical monogenics of order $k$ for each $k\in\mathbb{N}$, and generalize the Takenaka--Malmquist (TM) systems into higher dimensions. We note that the TM, or the rational orthogonal system, in one complex variable has attracted, and been attracting as well, considerable interest among analysts due to its theoretical involvements, through the Beurling theorem for instance, and applications. Extending TM systems to higher dimensional Clifford algebras was the main motivation of this study.

An adaptive TM system approximation theory for one complex variable has recently been established with applications (\cite{Qian1,QW}). The type of adaptive approximation has been generalized to some several complex variables and matrix-valued contexts with applications, see the works \cite{ACQS1,ACQS2,QWZ,QTW} by Alpay et al and Qian et al. As a crucial technical method the obtained Clifford orthogonalization process enables us to extend the one complex variable adaptive approximation theory to Clifford algebra (see \S 5), that as a consequence induces rational approximation to multi-real-variate functions through the Clifford imbedding.

The paper is organized as follows. In Section 2 we review Clifford algebra and Clifford analysis. In Section 3 we study orthogonalization of function systems in the right $\mathscr{A}_m$-module inner product space. In Section 4 we give definitions for TM systems and Blaschke products in general higher dimensions. In the last section we study adaptive approximation by Clifford TM systems in the unit ball and half space in higher dimensions.

\section{Preliminaries}
In this paper we work on the real Clifford algebra $\mathscr{A}_m$ that generated by an orthonormal basis $\{e_1,\ldots,e_m\}$ of $\mathbb{R}^m$ with
the (non-commutative) multiplication rule
$$e_ie_j+e_je_i=-2\delta_{ij}, i,j=1,\ldots,m,$$
where $\delta_{ij}$ equals 1 if $i=j$ and 0 otherwise. Each element $x$ in $\mathscr{A}_m$ is of the form
$$x=\sum_{T\in\mathcal{P}N}x_Te_T,$$
where $x_T\in \mathbb{R}$, $e_T=e_{i_1,\ldots,i_l}:=e_{i_1}e_{i_2}\cdots e_{i_l}$ is the basic element of $\mathscr{A}_m$, $T=\{i_1,\ldots,i_l\}$, $1\leq i_1<\ldots<i_l\leq m$, $\mathcal{P}N$ is the set consisting of all the ordered subsets of $\{1,\ldots,m\}$. In addition we set $x_\emptyset=x_0$, $e_\emptyset=e_0$, $e_0$ is identified with the multiplication unit ``$1$''. The multiplication of Clifford numbers is determined by the multiplication of the basic elements through linearity and the law of distribution. Let $e_A, e_B$ be any two basic elements in $\mathscr{A}_m$, their multiplication is defined by
$$e_Ae_B=(-1)^{\#(A\bigcap B)}(-1)^{p(A,B)}e_{A\triangle B},$$
where $p(A,B)=\sum_{j\in B}p(A,j)$, $p(A,j)=\#\{i\in A: i>j\}$, $A\triangle B=(A\setminus B)\bigcup(B\setminus A)$ is the symmetric difference of $A$ and $B$. Together with the multiplication, $\mathscr{A}_m$ is an associative algebra of dimension $2^m$.

For $x=\sum_Tx_Te_T\in\mathscr{A}_m$, we call $x_0$ the real part or scalar part of $x$, denote it by ${\rm Sc}~x$.
${\rm NSc}~x:=x-{\rm Sc}~x$ is then the non-scalar part of $x$. The norm and the conjugate of $x$ are defined by $|x|=(\sum_Tx_T^2)^{1/2}$ and $\overline{x}=\sum_Tx_T\overline{e_T}$ respectively, where
$\overline{e_T}=\overline{e_{i_l}}\cdots\overline{e_{i_2}}\,\overline{e_{i_1}}$, and $\overline{e_0}=e_0$, $\overline{e_i}=-e_i$ for $i\neq 0$. For any $x, y, z\in\mathscr{A}_m$, there holds $\overline{xy}=\overline{y}\, \overline{x}$, $(xy)z=x(yz)$, and $|xy|\leq 2^{m/2}|x||y|$. The real numbers, complex numbers and quaternions are special cases of Clifford algebra, i.e., we have $\mathscr{A}_0=\mathbb{R}$, $\mathscr{A}_1=\mathbb{C}$, and $\mathscr{A}_2=\mathbb{H}$.

For any $x\in\mathscr{A}_m$, we have ${\rm Sc}(\overline{x}x)={\rm Sc}(x\overline{x})=|x|^2$. If $x\in\mathscr{A}_m$ is of vector form, i.e., $x=\sum_{i=0}^mx_ie_i\in\mathbb{R}^{m+1}$, then obviously $\overline{x} x=x \overline{x}=|x|^2$, so, in the case, the inverse of $x$ is given by $x^{-1}=\overline{x}/|x|^2$ when $x\neq 0$. But for a general Clifford number $x$, the inverse of $x$ may not exist. That is to say, the Clifford algebra $\mathscr{A}_m$ is not a division algebra. Here we give a criterion for a Clifford number being invertible or not.

\begin{pro}\label{criterion}
Let $a\in\mathscr{A}_m$, the following conclusions are equivalent:
\begin{enumerate}
\item The equation $ax=0$ (or $xa=0$) has only zero solution $x=0$.
\item $a$ is invertible, i.e., there exists a unique $b\in\mathscr{A}_m$ such that $ab=ba=1$.
\item there exists $b\in\mathscr{A}_m$ such that $ab=1$ (or $ba=1$).
\end{enumerate}
\end{pro}

\begin{proof}
$(1)\Rightarrow(2)$: Note that the equation $ax=0$ can be written in the matrix form $AX=0$, where $A$ is a $2^m\times 2^m$ matrix associated to $a$, $X=(x_0,x_1,\ldots)^\top$ is the column vector whose components correspond to those of its algebraic representation. From this viewpoint, $ax=0$ has only zero solution $x=0$ means that the linear system of equations $AX=0$ has only zero solution $X=0$. So the matrix $A$ is invertible, and the equation $AX=(1,0,\ldots,0)^\top$ has a unique solution, given by $X=A^{-1}(1,0,\ldots,0)^\top$, which also gives the unique $b\in\mathscr{A}_m$ such that $ab=1$. To prove $ba=1$, now consider the equation $xa=0$, we get $x=xab=0$, so $xa=0$ has only zero solution $x=0$, similarly we get a unique $c\in\mathscr{A}_m$ such that $ca=1$, and $c=cab=b$, hence $ba=1$.

$(2)\Rightarrow(3)$: It is obvious.

$(3)\Rightarrow(1)$: Similar to the proof of $(1)\Rightarrow(2)$.
\end{proof}

Clifford analysis was founded by F. Brackx, R. Delanghe and F. Sommen et al (\cite{BDS}). As a generalization of complex analysis and quaternionic analysis into higher-dimensional spaces, Clifford analysis is a theory on Clifford monogenic functions. A function $f=\sum_{T\in\mathcal{P}N}f_Te_T$, defined on an open subset $\Omega$ of $\mathbb{R}^{m+1}$, taking values in $\mathscr{A}_m$, is said to be left monogenic on $\Omega$ if it satisfies the generalized Cauchy--Riemann equation
$$Df=\sum_{i=0}^me_i\frac{\partial f}{\partial x_i}=\sum_{i=0}^m\sum_{T\in\mathcal{P}N}\frac{\partial f_T}{\partial x_i}e_ie_T=0$$
for all $x\in\Omega$, where the Dirac operator $D$ is defined by
$$D=\frac{\partial}{\partial x_0}+\nabla=\sum_{i=0}^me_i\frac{\partial}{\partial x_i}.$$
If $f$ is left monogenic, then $\triangle f=\overline{D}(Df)=0$, so $f$ is harmonic. The set of all left monogenic functions on $\Omega$ constitutes a right $\mathscr{A}_m$-module.

If $f$ is left monogenic on $\Omega$ and continuous on $\overline{\Omega}$, then there holds Cauchy's integral formula
$$f(x)=\frac{1}{\omega_{m}}\int_{y\in\partial\Omega}E(y-x)n(y)f(y)dS,\quad x\in\Omega,$$
where $E(x)=\frac{\overline{x}}{|x|^{m+1}}$ is the Cauchy kernel, $\omega_m=2\pi^{\frac{m+1}{2}}/\Gamma(\frac{m+1}{2})$ is the
area of the unit sphere in $\mathbb{R}^{m+1}$, $n(y)$ is the outward-pointing unit normal vector and $dS$ is the surface area element on $\partial\Omega$.

For right monogenic functions there is a parallel theory.

\section{Orthogonalization in Hilbert modules over Clifford algebras}
In this section we discuss the orthogonalization problem of a right $\mathscr{A}_m$-module inner product space (for the case of left $\mathscr{A}_m$-modules one can similarly formulate). First we give some definitions (cf. \cite{BDS}).
\begin{defi}
A space $\mathscr{H}$ is called a right $\mathscr{A}_m$-module if the following conditions are fulfilled:
\begin{enumerate}
  \item $(\mathscr{H}, +)$ is an abelian group.
  \item A multiplication $(f, \lambda)\rightarrow f\lambda$ from $\mathscr{H}\times\mathscr{A}_m$ to $\mathscr{H}$ is defined such that for all $\lambda, \mu\in\mathscr{A}_m$ and $f, g\in\mathscr{H}$ there holds
    \begin{enumerate}
    \item[(1)] $f(\lambda+\mu)=f\lambda+f\mu$.
    \item[(2)] $f(\lambda\mu)=(f\lambda)\mu$.
    \item[(3)] $(f+g)\lambda=f\lambda+g\lambda$.
    \item[(4)] $fe_0=f$.
    \end{enumerate}
\end{enumerate}
\end{defi}

\begin{defi}
A space $\mathscr{H}$ is called a right $\mathscr{A}_m$-module normed space if the following conditions are fulfilled:
\begin{enumerate}
  \item $\mathscr{H}$ is a right $\mathscr{A}_m$-module.
  \item A norm $\|\cdot\|$ is defined on $\mathscr{H}$ such that
    \begin{enumerate}
    \item[(1)] $\|f\|\geq 0$ for all $f\in\mathscr{H}$, and $\|f\|=0$ if and only if $f=0$.
    \item[(2)] There is a real positive constant $C$ such that $\|f\lambda\|\leq C|\lambda|\|f\|$ for all $\lambda\in\mathscr{A}_m$, $f\in\mathscr{H}$, and $\|f\lambda\|=|\lambda|\|f\|$ for all $\lambda\in\mathbb{R}$, $f\in\mathscr{H}$.
    \item[(3)] $\|f+g\|\leq\|f\|+\|g\|$ for all $f, g\in\mathscr{H}$.
    \end{enumerate}
\end{enumerate}
\end{defi}

\begin{defi}
A space $\mathscr{H}$ (in which the element is also named ``function'') is called a right $\mathscr{A}_m$-module inner product space if the following conditions are fulfilled:
\begin{enumerate}
  \item $\mathscr{H}$ is a right $\mathscr{A}_m$-module.
  \item An inner product $(f, g)\rightarrow\langle f, g\rangle$ from $\mathscr{H}\times\mathscr{H}$ to $\mathscr{A}_m$ is defined such that for all $\lambda, \mu\in\mathscr{A}_m$ and $f, g, h\in\mathscr{H}$ there holds
    \begin{enumerate}
    \item[(1)] $\langle f, g\rangle=\overline{\langle g, f\rangle}$.
    \item[(2)] $\langle f\lambda+g\mu, h\rangle=\langle f, h\rangle\lambda+\langle g, h\rangle\mu$.
    \item[(3)] ${\rm Sc}\langle f, f\rangle\geq 0$, and ${\rm Sc}\langle f, f\rangle=0$ if and only if $f=0$.
    \item[(4)] $|{\rm Sc}\langle f, g\rangle|\leq\sqrt{{\rm Sc}\langle f, f\rangle}\sqrt{{\rm Sc}\langle g, g\rangle}$.
    \end{enumerate}
\end{enumerate}
\end{defi}

We have
\begin{pro}\label{IPS1}
Let $\mathscr{H}$ be a right $\mathscr{A}_m$-module inner product space, then for any $f, g\in\mathscr{H}$ we have
$$|\langle f, g\rangle|\leq 2^{\frac{m}{2}}\sqrt{{\rm Sc}\langle f, f\rangle}\sqrt{{\rm Sc}\langle g, g\rangle}.$$
In particular, $$|\langle f, f\rangle|\leq 2^{\frac{m}{2}}{\rm Sc}\langle f, f\rangle.$$
\end{pro}
\begin{proof}
Write $\langle f, g\rangle=\sum_{T\in\mathcal{P}N}\langle f, g\rangle_Te_T$, we get for every $T\in\mathcal{P}N$
\begin{align*}
\langle f, g\rangle_T^2&=({\rm Sc}(\overline{e_T}\langle f, g\rangle))^2
\\&=({\rm Sc}\langle f, ge_T\rangle)^2
\\&\leq({\rm Sc}\langle f, f\rangle)({\rm Sc}\langle ge_T, ge_T\rangle)
\\&=({\rm Sc}\langle f, f\rangle)({\rm Sc}(\overline{e_T}\langle g, g\rangle e_T))
\\&=({\rm Sc}\langle f, f\rangle)({\rm Sc}\langle g, g\rangle),
\end{align*}
so $|\langle f, g\rangle|=(\sum_{T\in\mathcal{P}N}\langle f, g\rangle_T^2)^{1/2}\leq 2^{\frac{m}{2}}\sqrt{{\rm Sc}\langle f, f\rangle}\sqrt{{\rm Sc}\langle g, g\rangle}$.
\end{proof}

\begin{pro}\label{IPS2}
Every right $\mathscr{A}_m$-module inner product space $\mathscr{H}$ is a right $\mathscr{A}_m$-module normed space with the induced norm $\|f\|:=\sqrt{{\rm Sc}\langle f, f\rangle}$ for $f\in\mathscr{H}$.
\end{pro}
\begin{proof}
For any $\lambda\in\mathscr{A}_m$ and $f, g\in\mathscr{H}$,
\begin{align*}
\|f\lambda\|&=\sqrt{{\rm Sc}\langle f\lambda, f\lambda\rangle}
\\&=\sqrt{{\rm Sc}(\overline{\lambda}\langle f, f\rangle\lambda)}
\\&=\sqrt{{\rm Sc}(\lambda\overline{\lambda}\langle f, f\rangle)}
\\&\leq\sqrt{|\lambda\overline{\lambda}||\langle f, f\rangle|}
\\&\leq\sqrt{2^{\frac{m}{2}}|\lambda|^2\cdot2^{\frac{m}{2}}\|f\|^2}
\\&=2^{\frac{m}{2}}|\lambda|\|f\|,
\end{align*}
and $\|f+g\|^2={\rm Sc}\langle f+g, f+g\rangle={\rm Sc}(\langle f, f\rangle+2\langle f, g\rangle+\langle g, g\rangle)\leq(\|f\|+\|g\|)^2$.
\end{proof}

A complete right $\mathscr{A}_m$-module normed space is called a right $\mathscr{A}_m$-module Banach space, and a complete right $\mathscr{A}_m$-module inner product space is called a right $\mathscr{A}_m$-module Hilbert space. The case for the left $\mathscr{A}_m$-module can be similarly formulated.

\begin{lem}\label{completeness}
If $\mathscr{H}$ is a right $\mathscr{A}_m$-module inner product space, then for any function $f\in\mathscr{H}$, $\{fc: c\in\mathscr{A}_m\}$ is a close subspace of $\mathscr{H}$.
\end{lem}
\begin{proof}
Our goal is to show that if $\|fc_N-fc_M\|\rightarrow 0$ ($N, M\rightarrow\infty$),
then there exists $c\in\mathscr{A}_m$ such that $\|fc_N-fc\|\rightarrow 0$ as $N\rightarrow\infty$.
Because $$\|fc_N-fc_M\|^2={\rm Sc}\langle f(c_N-c_M), f(c_N-c_M)\rangle={\rm Sc}(\overline{(c_N-c_M)}\langle f,f\rangle(c_N-c_M))\geq 0,$$
$\|fc_N-fc_M\|^2$ can be seen as a positive semidefinite quadratic form of $c_N-c_M$.
Now we treat $c_N-c_M$ as a column vector whose $i$-th component coincides with the $i$-th component of its algebraic form,
and denote by $A$ the real symmetric matrix associated to the quadratic form $\|fc_N-fc_M\|^2$, which is determined by $\langle f,f\rangle$.
Then we have $$\|fc_N-fc_M\|^2=(c_N-c_M)^\top A(c_N-c_M).$$
Let $\Gamma$ be the orthogonal matrix such that $\Gamma^\top A\Gamma$ is a diagonal matrix with the positive diagonal entries being $\lambda_{i_1},\ldots,\lambda_{i_k}$, and write
$c_N-c_M=\Gamma(d_N-d_M)$, then
\begin{align*}
\|fc_N-fc_M\|^2&=(d_N-d_M)^\top\Gamma^\top A\Gamma(d_N-d_M)
\\&=\lambda_{i_1}(d_{N_{i_1}}-d_{M_{i_1}})^2+\ldots+\lambda_{i_k}(d_{N_{i_k}}-d_{M_{i_k}})^2\rightarrow 0,
\end{align*}
which means $d_{N_{i_1}}-d_{M_{i_1}}\rightarrow 0,\ldots,d_{N_{i_k}}-d_{M_{i_k}}\rightarrow 0$ as $N, M\rightarrow\infty$.
By the completeness of the real numbers, there exists $d_{i_1},\ldots,d_{i_k}$ such that
$d_{N_{i_1}}-d_{i_1}\rightarrow 0,\ldots,d_{N_{i_k}}-d_{i_k}\rightarrow 0$ as $N\rightarrow\infty$. Now let
$$d=(0,\ldots,0,d_{i_1},0,\ldots,0,d_{i_2},0,\ldots,0,d_{i_k},0,\ldots,0)^\top,$$
$c=\Gamma d$, then
\begin{align*}
\|fc_N-fc\|^2&=(d_N-d)^\top\Gamma^\top A\Gamma(d_N-d)
\\&=\lambda_{i_1}(d_{N_{i_1}}-d_{i_1})^2+\ldots+\lambda_{i_k}(d_{N_{i_k}}-d_{i_k})^2\rightarrow 0
\end{align*}
as $N\rightarrow\infty$.
\end{proof}

\begin{lem}\label{OP}
If $\mathscr{H}$ is a right $\mathscr{A}_m$-module inner product space,
then for any functions $\alpha, \beta\in\mathscr{H}$, the orthogonal projection
of $\alpha$ onto the subspace spanned by $\beta$ uniquely exists,
denoted by $\mathscr{P}_{\overline{\mbox{\scriptsize span}}\{\beta\}}\alpha$.
\end{lem}
\begin{proof}
The purpose is to show that there exists a unique $\beta c$ such that
$$\|\alpha-\beta c\|=\inf_{c'\in\mathscr{A}_m}\|\alpha-\beta c'\|,$$
and such $\beta c$ satisfies
$$\langle \alpha-\beta c, \beta\rangle=0.$$
Let $d=\inf_{c'\in\mathscr{A}_m}\|\alpha-\beta c'\|$, then for any $N\in\mathbb{N}_+$ there exists $c_N\in\mathscr{A}_m$ such that
$d\leq\|\alpha-\beta c_N\|\leq d+\frac{1}{N}$. By the parallelogram identity,
\begin{align*}
\|\beta c_N-\beta c_M\|^2&=\|(\alpha-\beta c_N)-(\alpha-\beta c_M)\|^2
\\&=2(\|\alpha-\beta c_N\|^2+\|\alpha-\beta c_M\|^2)-4\|\alpha-\beta\frac{c_N+c_M}{2}\|^2
\\&\leq 2((d+\frac{1}{N})^2+(d+\frac{1}{M})^2)-4d^2\rightarrow 0
\end{align*}
as $N, M\rightarrow\infty$. By Lemma \ref{completeness}, $\{\beta c_N\}_{N=1}^\infty$ has a limit $\beta c$,
so by the continuity of the norm we get $\|\alpha-\beta c\|=d$. To prove the uniqueness,
suppose there is another $\beta\widetilde{c}$ satisfying $\|\alpha-\beta\widetilde{c}\|=d$, then
\begin{align*}
\|\beta c-\beta\widetilde{c}\|^2&=2(\|\alpha-\beta c\|^2+\|\alpha-\beta\widetilde{c}\|^2)-4\|\alpha-\beta\frac{c+\widetilde{c}}{2}\|^2
\\&\leq 4d^2-4d^2=0,
\end{align*}
which implies $\beta c=\beta\widetilde{c}$. Finally, we turn to show that $\langle \alpha-\beta c, \beta\rangle=0$.
For any $x\in\mathbb{R}$, we have
\begin{align*}
d^2&\leq\|\alpha-\beta c-\beta x\|^2
\\&={\rm Sc}\langle\alpha-\beta c-\beta x, \alpha-\beta c-\beta x\rangle
\\&=\|\alpha-\beta c\|^2-2x{\rm Sc}\langle \alpha-\beta c, \beta\rangle+x^2\|\beta\|^2
\\&=d^2-2x{\rm Sc}\langle \alpha-\beta c, \beta\rangle+x^2\|\beta\|^2.
\end{align*}
So $$-2x{\rm Sc}\langle \alpha-\beta c, \beta\rangle+x^2\|\beta\|^2\geq 0$$
for all $x\in\mathbb{R}$, which implies
$${\rm Sc}\langle \alpha-\beta c, \beta\rangle=0.$$
For each $T\in\mathcal{P}N$, replace $\beta x$ by $\beta e_T x$ and repeat the above discussions
we see that every component of $\langle \alpha-\beta c, \beta\rangle$ equals $0$. Hence $\langle \alpha-\beta c, \beta\rangle=0$.
\end{proof}

\begin{rem}\normalfont
In the above proof the orthogonal projection $\beta c$ is unique, but $c\in\mathscr{A}_m$ may not be unique,
that is different from the case of complex inner product space.
\end{rem}

As a consequence of Lemma \ref{OP} we have
\begin{theo}\label{main}
Let $\{\alpha_n\}_{n=1}^\infty$ be a sequence of functions in a right $\mathscr{A}_m$-module inner product space $\mathscr{H}$. Set
\begin{align*}
\beta_1&=\alpha_1,\\
\beta_2&=\alpha_2-\mathscr{P}_{\overline{\mbox{\scriptsize span}}\{\beta_1\}}\alpha_2,\\
&\,\,\,\vdots\\
\beta_n&=\alpha_n-\sum_{i=1}^{n-1}\mathscr{P}_{\overline{\mbox{\scriptsize span}}\{\beta_i\}}\alpha_n,\\
&\,\,\,\vdots
\end{align*}
then $\{\beta_n\}_{n=1}^\infty$ is an orthogonal system of functions in $\mathscr{H}$.
\end{theo}

As an application, we now consider the inner spherical monogenics of order $k$ ($k\in\mathbb{N}$) in Clifford analysis
(play an analogous role as the powers of the complex variable $z$), denoted by
$$\mathcal{M}_k=\{V_{l_1,\ldots,l_k}: (l_1,\ldots,l_k)\in\{1,\ldots,m\}^k\},$$
where by definition $V_0(x)=e_0$,
$$V_{l_1,\ldots,l_k}=\frac{1}{k!}\sum_{\pi(l_1,\ldots,l_k)}z_{l_1}\ldots z_{l_k},$$
in which the sum runs over all distinguishable permutations of $l_1,\ldots,l_k$,
and the hyper-complex variables
$$z_l=x_le_0-x_0e_l,\quad l=1,\ldots,m.$$
For $f, g\in\bigcup_{k\in\mathbb{N}}\mathcal{M}_k$, the inner product is defined by
$$\langle f, g\rangle:=\frac{1}{\omega_m}\int_{S^m}\overline{g}fdS,$$
with the induced norm
$$\|f\|:=({\rm Sc}\langle f, f\rangle)^{1/2}=\left(\frac{1}{\omega_m}\int_{S^m}|f|^2dS\right)^{1/2},$$
where $S^m$ is the unit sphere in $\mathbb{R}^{m+1}$ centered at the origin, $dS$ is the surface area element on $S^m$.

Inner spherical monogenics of different orders are mutually orthogonal, but for a fixed order $k$, there are $\binom{m+k-1}{k}$ elements in
$\mathcal{M}_k$ being not necessarily mutually orthogonal. So it is natural to ask for the construction of the orthonormal basis of $\mathcal{M}_k$.
The existence of the orthonormal basis of $\mathcal{M}_k$ was proved in \cite{BDS} by induction, but with no concrete expressions. By Theorem \ref{main} we can now immediately give the explicit orthogonal formulas. More precisely, we have
\begin{theo}\label{orthbasis}
Rearrange the elements in $\mathcal{M}_k$ by writing $\mathcal{M}_k=\{V_1, V_2, \ldots, V_n\}$, where $n=\binom{m+k-1}{k}$,
then $\langle V_1, V_1\rangle$ is invertible. Let $U_1=V_1$, then
$$\mathscr{P}_{\overline{\mbox{\scriptsize span}}\{U_1\}}V_2=U_1\langle U_1,
U_1\rangle^{-1}\langle V_2, U_1\rangle.$$
Let
$$U_2=V_2-\mathscr{P}_{\overline{\mbox{\scriptsize span}}\{U_1\}}V_2,$$
then $\langle U_2, U_2\rangle$ is invertible and $\langle U_2, U_1\rangle=0$. In general, let
$$U_j=V_j-\sum_{i=1}^{j-1}\mathscr{P}_{\overline{\mbox{\scriptsize span}}\{U_i\}}U_j
=V_j-\sum_{i=1}^{j-1}U_i\langle U_i, U_i\rangle^{-1}\langle V_j,
U_i\rangle,\,\mbox{for } j\leq n,$$ then $\langle U_j, U_j\rangle$ is invertible for each $j\leq n$, and $\langle U_j, U_l\rangle=0$ for $j\neq l$.
So $\{U_1,\ldots,U_n\}$ consists an orthogonal basis of $\mathcal{M}_k$.
\end{theo}
\begin{proof}
Consider the equation $\langle U_j, U_j\rangle c=0$ in $c$, then $\overline{c}\langle U_j, U_j\rangle c=\langle U_jc, U_jc\rangle=0$, so $\|U_jc\|^2={\rm Sc}\langle U_jc, U_jc\rangle=0$, which gives $U_jc=0$. Note that $U_jc$ is a linear combination of $V_1,\ldots,V_j$ with the coefficient of $V_j$ being $c$, by the uniqueness of the Taylor series we get $c=0$. By Proposition \ref{criterion} we conclude that $\langle U_j, U_j\rangle$ is invertible. The orthogonality $\langle U_i, U_j\rangle=0$ for $i\neq j$ can be directly verified.
\end{proof}

\section{Takenaka--Malmquist systems in higher dimensions}
Denote by $B^{m+1}$ the unit ball in $\mathbb{R}^{m+1}$ centered at the origin, $B^{m+1}=\{x\in \mathbb{R}^{m+1}: |x|<1\}$, $S^m=\partial B^{m+1}$. The monogenic Hardy space $\mathcal{H}^2(B^{m+1})$ consists of all left monogenic functions $f$ on $B^{m+1}$ that satisfy
$$\|f\|:=\sup_{0<r<1}\left(\frac{1}{\omega_m}
\int_{\eta\in S^m}|f(r\eta)|^2dS\right)^{1/2}<\infty.$$
For $f,g\in\mathcal{H}^2(B^{m+1})$, their Clifford number-valued inner product is defined by
$$\langle f, g\rangle:=\frac{1}{\omega_m}\int_{\eta\in S^m}\overline{g(\eta)}f(\eta)dS,$$
where $f(\eta)$ and $g(\eta)$ ($\eta\in S^m$) are respectively the non-tangential boundary limit of $f$ and $g$.
We have
$$\|f\|=({\rm Sc}\langle f, f\rangle)^{1/2}=\left(\frac{1}{\omega_m}\int_{\eta\in S^m}|f(\eta)|^2dS\right)^{1/2}.$$
$\mathcal{H}^2(B^{m+1})$ is a right $\mathscr{A}_m$-module Hilbert space.

Let $a\in B^{m+1}$,
$$S_a(x)=\frac{\overline{1-\overline{a}x}}
{|1-\overline{a}x|^{m+1}},\quad x\in B^{m+1}$$
be the Szeg\"{o} kernel for $B^{m+1}$. For any multi-index $k=(k_0,k_1,\ldots,k_m)\in\mathbb{N}^{m+1}$
and any $f\in\mathcal{H}^2(B^{m+1})$, by Cauchy's integral formula we have
\begin{equation}\label{Cauchy}
\langle f,\partial^k_aS_a\rangle=(\partial^k_xf)(a),
\end{equation}
where $\partial^k_xf=\frac{\partial^{|k|}f}{\partial x_0^{k_0}\partial
x_1^{k_1}\cdots\partial x_m^{k_m}}$, $|k|=\sum_{i=0}^mk_i$.

Let $\{a_n\}_{n=1}^{\infty}$ be a sequence of Clifford numbers taking values in $B^{m+1}$.
If $a_n$ ($n\in\mathbb{N}_+$) are distinct from each other, then we have
\begin{theo}
The GS orthogonalization process
\begin{equation*}
\left\{\begin{aligned}
&T_{a_1}:=S_{a_1},\\
&T_{a_1,\ldots,a_n}:=S_{a_n}-\sum_{i=1}^{n-1}T_{a_1,\ldots,a_i}
\langle T_{a_1,\ldots,a_i}, T_{a_1,\ldots,a_i}\rangle^{-1}\langle S_{a_n}, T_{a_1,\ldots,a_i}\rangle,\,n\geq 2
\end{aligned}\right.
\end{equation*}
is realizable.
\end{theo}
\begin{proof}
To show that $\langle T_{a_1,\ldots,a_n}, T_{a_1,\ldots,a_n}\rangle$ is invertible, consider the equation
$$\langle T_{a_1,\ldots,a_n}, T_{a_1,\ldots,a_n}\rangle c=0.$$
By the same argument as that in the proof of Theorem \ref{orthbasis}, we have
\begin{equation}\label{identity}
T_{a_1,\ldots,a_n}(x)c=S_{a_n}(x)c+\sum_{i=1}^{n-1}S_{a_i}(x)c_i\equiv 0
\end{equation}
for some Clifford numbers $c_1,\ldots,c_{n-1}\in\mathscr{A}_m$ and $x\in B^{m+1}$.
Since $S_{a_1},\ldots S_{a_n}$ are of different poles outside the unit sphere, we can show that
$$c=c_1=\ldots=c_{n-1}=0.$$
To be specific, firstly by the uniqueness theorem of monogenic functions we can extend the identity \eqref{identity} to $\mathbb{R}^{m+1}\setminus\{\frac{a_1}{|a_1|^2},\ldots,\frac{a_n}{|a_n|^2}\}$. After multiplying \eqref{identity} by
$$(1-\overline{a_n}x)|1-\overline{a_n}x|^{m-1}$$
from the left-hand side we get
$$c+(1-\overline{a_n}x)|1-\overline{a_n}x|^{m-1}\sum_{i=1}^{n-1}S_{a_i}(x)c_i\equiv 0$$
for all $x\in\mathbb{R}^{m+1}\setminus\{\frac{a_1}{|a_1|^2},\ldots,\frac{a_n}{|a_n|^2}\}$.
Letting $x\rightarrow \frac{a_n}{|a_n|^2}$ we obtain
$c=0$, which implies that $\langle T_{a_1,\ldots,a_n}, T_{a_1,\ldots,a_n}\rangle^{-1}$ exists by Proposition \ref{criterion}.
\end{proof}

\begin{rem}\normalfont
We have checked by calculations that $\langle T_{a_1,\ldots,a_n}, T_{a_1,\ldots,a_n}\rangle$ is a positive real number for $n\leq 5$. We conjecture that it holds for all $n\in\mathbb{N}_+$.
\end{rem}

Hence, $\{B_n\}:=\{B_{a_1,\ldots,a_n}\}:=\{\frac{T_{a_1,\ldots,a_n}}{||T_{a_1,\ldots,a_n}||}\}_{n=1}^\infty$ becomes an orthonormal system for $\mathcal{H}^2(B^{m+1})$.

But if at least two of the parameters are the same, for example,
$a_2$ equals $a_1$, then obviously $T_{a_1,a_2}=T_{a_1,a_1}=0$. At this case we
interpret $B_2$ as $\lim_{\rho\rightarrow 0^+}B_{a_1,b}$ (\cite{QSW}), where
$b=a_1+\rho\omega$, $\omega=\cos\theta_1+\sin\theta_1\cos\theta_2e_1+\sin\theta_1\sin\theta_2\cos\theta_3e_2+\ldots+\sin\theta_1\sin\theta_2\cdots\sin\theta_me_m$,
and $\theta_1, \theta_2, \ldots, \theta_{m-1}\in[0,\pi]$, $\theta_m\in[0,2\pi]$. More precisely,
\begin{align*}
B_2:=&\lim_{\rho\rightarrow 0^+}B_{a_1,b}
\\=&\lim_{\rho\rightarrow 0^+}\frac{T_{a_1,b}}{\|T_{a_1,b}\|}
\\=&\lim_{\rho\rightarrow 0^+}\frac{T_{a_1,b}-T_{a_1,a_1}}{\|T_{a_1,b}-T_{a_1,a_1}\|}
\\=&\lim_{\rho\rightarrow 0^+}\frac{\frac{T_{a_1,b}-T_{a_1,a_1}}{\rho}}
{\|\frac{T_{a_1,b}-T_{a_1,a_1}}{\rho}\|}
\\=&\frac{\nabla_\omega T_{a_1,y}|_{y=a_1}}{\|\nabla_\omega T_{a_1,y}|_{y=a_1}\|}
\\=&\frac{\nabla_\omega S_y|_{y=a_1}-T_{a_1}\langle T_{a_1}, T_{a_1}\rangle^{-1}\langle\nabla_\omega S_y|_{y=a_1}, T_{a_1}\rangle}
{\|\nabla_\omega S_y|_{y=a_1}-T_{a_1}\langle T_{a_1}, T_{a_1}\rangle^{-1}\langle\nabla_\omega S_y|_{y=a_1}, T_{a_1}\rangle\|}
\end{align*}
where
$\nabla_\omega S_y=\frac{\partial S_y}{\partial y_0}\cos\theta_1+\frac{\partial S_y}{\partial y_1}\sin\theta_1\cos\theta_2
+\frac{\partial S_y}{\partial y_2}\sin\theta_1\sin\theta_2\cos\theta_3+\ldots+\frac{\partial S_y}{\partial y_m}
\sin\theta_1\sin\theta_2\cdots\sin\theta_m$ is the directional derivative of $S_y$ with respect to $y$. In other words, when
$a_2=a_1$, $B_2$ is interpreted as the orthonormalization of $T_{a_1}$ and $\nabla_\omega S_y|_{y=a_1}$.

We further note that as a function of $y$, $S_y$ satisfies $S_y\overline{D}=0$, which implies
that $\frac{\partial S_y}{\partial y_0}, \frac{\partial S_y}{\partial y_1},\ldots,\frac{\partial S_y}{\partial y_m}$ are linear dependent
in $\mathcal{H}^2(B^{m+1})$. Hence, if the multiplicity of the parameter $a_n$ (we call the cardinal number of the set
$\{j: a_j=a_n, j\leq n\}$ the multiplicity of $a_n$ and denote it by $m(a_n)$) is greater than $m+1$, then the second order partial
derivatives of $S_y$ at the point $a_n$ should be involved in the orthogonalization process. In general, when
$m(a_n)>\sum_{i=0}^{k-1}\binom{i+m-1}{m-1}=\binom{k+m-1}{m}$, then the
$k$-th order partial derivatives of $S_y$ at the point $a_n$ must appear.

Observe that in complex analysis the TM systems for the unit disc and upper half space can be generated by
Szeg\"{o} or higher order Szeg\"{o} kernels through GS orthogonalization process (\cite{WQ2}), heuristically, we propose the following definition.
\begin{defi}
We call $\{B_n\}_{n=1}^\infty$ the Takenaka--Malmquist system for $B^{m+1}$. If the $k$-th parameter $a_k=0$, then
$B_k$ is called a Blaschke product of order $k-1$ for $B^{m+1}$.
\end{defi}

By the orthogonality of $\{B_n\}_{n=1}^\infty$ and the reproducing property of the Szeg\"{o} kernel we easily get the following property
similar to the complex TM systems:
\begin{pro}
For any $B_{a_1,\ldots,a_n}$ in the TM system, $a_i$ ($i\leq n-1$) is a zero point of $B_{a_1,\ldots,a_n}$ with multiplicity $m(a_i)$.
\end{pro}

For the cases of half space and general domains (provided that the Szeg\"{o} kernels exist) we have similar results. Let
$\mathbb{R}^{m+1}_+:=\{x\in \mathbb{R}^{m+1}: {\rm Sc}\,x>0\}$ be the half space in $\mathbb{R}^{m+1}$, the Szeg\"{o} kernel we use for $\mathbb{R}^{m+1}_+$
is
$$S_a(x)=\frac{\overline{x+\overline{a}}}
{|x+\overline{a}|^{m+1}},\quad x, a\in\mathbb{R}^{m+1}_+.$$

\section{Adaptive Clifford TM system approximation}
Let us first have a brief review of the one complex variable adaptive TM system approximation. Consider the complex Hardy space $\mathcal{H}^2(\D)$, where $\D$ denotes the unit disc in the complex plane. For $f\in\mathcal{H}^2(\D)$, there exists an adaptive TM system approximation of $f,$ expressed as
$$f=\sum_{k=1}^\infty\langle f, B_k\rangle B_k=\sum_{k=1}^\infty\langle f_k, B_k\rangle B_k=\sum_{k=1}^\infty\langle g_k, e_{a_k}\rangle B_k,$$
where $\{B_k(z)\}_{k=1}^\infty$ is the Takenaka--Malmquist (TM) system on $\D$ determined by a sequence $\{a_k\}_{k=1}^\infty$ in $\D$
being specially selected according to the Maximal Selection Principle (see below) of the context,
$$B_k(z)=\frac{\sqrt{1-|a_k|^2}}{z-a_k}\prod_{l=1}^{k}\frac{z-a_l}{1-\overline{a_l}z},$$
$f_k$ is the $k$-th \emph{standard remainder}, defined by
$$f_k:=f-\sum_{l=1}^{k-1}\langle f, B_l\rangle B_l=f-\sum_{l=1}^{k-1}\langle g_l, e_{a_l}\rangle B_l,$$
and $g_l$ is the $l$-th \emph{reduced remainder}, defined by
$$g_l(z)=f_l(z)\prod_{j=1}^{l-1}\frac{1-\overline{a_j}z}{z-a_j},$$
and
$$e_{a_l}(z)=\frac{\sqrt{1-|a_l|^2}}{1-\overline{a_l}z}$$
being the normalized Szeg\"{o} kernel of $\D$, that plays the role as reproducing kernel of the Hilbert space $\mathcal{H}^2(\D)$. When
$a_1,a_2,\ldots$ are mutually different, $B_1,B_2,\ldots$ are consecutively GS orthonormalizations of $e_{a_1},e_{a_2},\ldots$; and if
$a_1,a_2,\ldots$ have multiples, in the GS process $e_{a_1},e_{a_2},\ldots$ are replaced by the so called higher order Szeg\"o kernels being corresponding derivatives of the Szeg\"o kernels (\cite{QW}).

The approximation is said to be adaptive because for each $k$ the parameter $a_k$ in the defined TM system is adaptively chosen to best match the $k$-th reduced reminder $g_k$ which amounts to selecting $a_k$ according to the Maximum Selection Principle
$$a_k=\mathop{\arg\max}\limits_{a\in D}|\langle g_k, e_a\rangle|^2=\mathop{\arg\max}\limits_{a\in D}(1-|a|^2)|g_k(a)|^2.$$
Note that if all the parameters are zero, then the TM system reduces to the half Fourier system. If $a_k=0$, then $B_k$ becomes a Blaschke product. If the first parameter $a_1$ is chosen to be zero, then we get an adaptive mono-components decomposition, i.e., every $B_k$ is a mono-component, or, in other words, each $B_k$ possesses a non-negative analytic instantaneous frequency function. The case for the upper half plane is similar, and is discussed in \cite{Qian1}.

The advantages of adaptive TM system approximation over the usual greedy algorithms (\cite{DMA,Tem1,Tem2}) include that at each step the former achieves the optimal energy pursuit and at the same time produces an added new term to possess positive analytic frequency. Such optimal matching pursuit method has been extended to general Hilbert spaces with a dictionary satisfying the so called boundary vanishing condition (\cite{Qian2}).

Generalization of adaptive TM system approximation into multivariate functions has been following two routes. One is for several complex variables (\cite{ACQS1, Qian2}), the other is for several real variables in the frame work of quaternionic (\cite{QSW}) and Clifford analysis. In the context of several complex variables, in \cite{ACQS1},  the Drury--Arveson space of functions analytic in the unit ball of $\mathbb{C}^N$ is discussed. In \cite{Qian2}, in the context of the $n$-torus $T^n$, two different approaches are discussed, of which one uses product-TM systems  and the other uses the product-Szeg\"{o} kernel dictionaries. The several complex variables contexts are commutative, with the invertibility inherited from the complex numbers, that all together make the usual GS orthogonalization process applicable.   The case for matrix-valued functions was studied in \cite{ACQS2}.

Since the Euclidean space $\mathbb{R}^n$ can be naturally embedded into quaternions or a Clifford algebra, it is natural to perform quaternionic or Clifford GS orthogonalization process in order to construct an analogous adaptive approximation theory. Without a GS orthogonalization process and without special functions playing a similar rope as Blaschke products, what have been achieved are only the greedy type algorithms (\cite{QWY,WQ}). The Clifford TM system constructed in \S 4 plays a definitive role in adaptive TM system approximation in several real variables in the frame work of Clifford monogenic functions.

Let $f\in\mathcal{H}^2(B^{m+1})$. We associate $f$ with the Fourier-type series
$$f(x)\sim\sum_{n=1}^\infty B_n(x)c_n,$$
where the coefficients $c_n$'s are given by
$$c_1=\langle B_1, B_1\rangle^{-1}\langle f, B_1\rangle=(1-|a_1|^2)^{\frac{m}{2}}f(a_1),$$
and for $n\geq 2$,
\begin{align*}
c_n=&\langle B_n, B_n\rangle^{-1}\langle f, B_n\rangle
\\=&\langle B_n, B_n\rangle^{-1}\left\langle f, \frac{S_{a_n}-\sum_{i=1}^{n-1}T_{a_1,\ldots,a_i}
\langle T_{a_1,\ldots,a_i}, T_{a_1,\ldots,a_i}\rangle^{-1}\langle S_{a_n}, T_{a_1,\ldots,a_i}\rangle}
{\|T_{a_1,\ldots,a_n}\|}\right\rangle
\\=&\langle B_n, B_n\rangle^{-1}\frac{\left\langle f-\sum_{i=1}^{n-1}T_{a_1,\ldots,a_i}\langle T_{a_1,\ldots,a_i}, T_{a_1,\ldots,a_i}\rangle^{-1}
\langle f, T_{a_1,\ldots,a_i}\rangle, S_{a_n}\right\rangle}{\|T_{a_1,\ldots,a_n}\|}.
\end{align*}
Let
\begin{align*}
f_n(x)&=f(x)-\sum_{i=1}^{n-1}T_{a_1,\ldots,a_i}(x)\langle T_{a_1,\ldots,a_i}, T_{a_1,\ldots,a_i}\rangle^{-1}
\langle f, T_{a_1,\ldots,a_i}\rangle
\\&=f(x)-\sum_{i=1}^{n-1}B_i(x)\langle B_i, B_i\rangle^{-1}\langle f, B_i\rangle.
\end{align*}
If $m(a_n)=1$, then
\begin{align}
c_n=\frac{\langle B_n, B_n\rangle^{-1}}{\|T_{a_1,\ldots,a_n}\|}f_n(a_n)
=\|T_{a_1,\ldots,a_n}\|\langle T_{a_1,\ldots,a_n}, T_{a_1,\ldots,a_n}\rangle^{-1}f_n(a_n),
\end{align}
\begin{align}
\|B_nc_n\|^2&={\rm Sc}\langle B_n c_n, B_n c_n\rangle\nonumber
\\&={\rm Sc}\left\langle B_n\frac{\langle B_n, B_n\rangle^{-1}}{\|T_{a_1,\ldots,a_n}\|}f_n(a_n),
B_n\frac{\langle B_n, B_n\rangle^{-1}}{\|T_{a_1,\ldots,a_n}\|}f_n(a_n)\right\rangle\nonumber
\\&={\rm Sc}(\overline{f_n(a_n)}\frac{\langle B_n, B_n\rangle^{-1}}{\|T_{a_1,\ldots,a_n}\|^2}f_n(a_n))\nonumber
\\&={\rm Sc}(\overline{f_n(a_n)}\langle T_{a_1,\ldots,a_n}, T_{a_1,\ldots,a_n}\rangle^{-1}f_n(a_n))\nonumber
\\&={\rm Sc}((1-|a_n|^2)^{m}\overline{f_n(a_n)}((1-|a_n|^2)^{m}\langle T_{a_1,\ldots,a_n}, T_{a_1,\ldots,a_n}\rangle)^{-1}f_n(a_n)),\label{energy}
\end{align}
and
\begin{align}
&(1-|a_n|^2)^{m}\langle T_{a_1,\ldots,a_n}, T_{a_1,\ldots,a_n}\rangle\nonumber
\\=&(1-|a_n|^2)^{m}\left(\langle S_{a_n}, S_{a_n}\rangle
-\sum_{i=1}^{n-1}\langle T_{a_1,\ldots,a_i}, S_{a_n}\rangle\langle T_{a_1,\ldots,a_i}, T_{a_1,\ldots,a_i}\rangle^{-1}
\overline{\langle T_{a_1,\ldots,a_i}, S_{a_n}\rangle}\right)\nonumber
\\=&1-(1-|a_n|^2)^{m}\sum_{i=1}^{n-1}T_{a_1,\ldots,a_i}(a_n)\langle T_{a_1,\ldots,a_i}, T_{a_1,\ldots,a_i}\rangle^{-1}
\overline{T_{a_1,\ldots,a_i}(a_n)}.\label{inner}
\end{align}
If $m(a_n)>1$, $c_n$ and $\|B_nc_n\|^2$ are taken in the limit sense as before.

\begin{lem}\label{limit1}
Let $a_1,\ldots,a_{n-1}\in B^{m+1}$ be fixed, $a=|a|\xi=r\xi$, then
$$\lim_{r\rightarrow 1^-}\|B_{a_1,\ldots,a_{n-1},a}\langle B_{a_1,\ldots,a_{n-1},a}, B_{a_1,\ldots,a_{n-1},a}\rangle^{-1}\langle f, B_{a_1,\ldots,a_{n-1},a}\rangle\|^2=0$$
holds uniformly in $|\xi|=1$.
\end{lem}
\begin{proof}
Note that when $r\rightarrow 1^-$, $a$ must be different from $a_i$ ($i\leq n-1$), then \eqref{inner} clearly shows that
$$\lim_{r\rightarrow 1^-}(1-|a|^2)^{m}\langle T_{a_1,\ldots,a_{n-1},a}, T_{a_1,\ldots,a_{n-1},a}\rangle=1.$$ On the other hand, according
to Lemma 3.2 in \cite{QWY} we have
$$\lim_{r\rightarrow 1^-}(1-|a|^2)^{\frac{m}{2}}f_n(a)=0$$
uniformly in $|\xi|=1$.
So from \eqref{energy} we immediately get the desired result.
\end{proof}

Lemma \ref{limit1} implies

\begin{theo}[Maximum Selection Principle]
For any $f\in\mathcal{H}^2(B^{m+1})$ and any fixed $a_1,\ldots,a_{n-1}\in B^{m+1}$, there exist an $a_n\in B^{m+1}$ such that
\begin{align}\label{MSP}
&\|B_{a_1,\ldots,a_{n-1},a_n}\langle B_{a_1,\ldots,a_{n-1},a_n}, B_{a_1,\ldots,a_{n-1},a_n}\rangle^{-1}
\langle f, B_{a_1,\ldots,a_{n-1},a_n}\rangle\|\nonumber\\=&\sup_{a\in B^{m+1}}\|B_{a_1,\ldots,a_{n-1},a}\langle B_{a_1,\ldots,a_{n-1},a}, B_{a_1,\ldots,a_{n-1},a}\rangle^{-1}\langle f, B_{a_1,\ldots,a_{n-1},a}\rangle\|
\end{align}
\end{theo}

The maximum selection principle enables us to obtain the best approximation to $f$ step by step,
by choosing a suitable parameter $a_n$ at the $n$-th step such that the energy of the $n$-th term
$B_n\langle B_n, B_n\rangle^{-1}\langle f, B_n\rangle$ attains its maximum,
or equivalently, making the energy of the residue $f_n$ attain its minimum,
so that the adaptive Fourier series associated to $f$ converges in a fast way.
Note that the choice of $a_n$ in \eqref{MSP} may not be unique.

We now proceed to prove the convergence of the adaptive Fourier series. First we show a technical lemma.

\begin{lem}\label{ineq1}
For any $a_1,\ldots,a_n\in B^{m+1}$, we have
\begin{align*}
\|B_{a_1,\ldots,a_n}\langle B_{a_1,\ldots,a_n}, B_{a_1,\ldots,a_n}\rangle^{-1}\langle f, B_{a_1,\ldots,a_n}\rangle\|
&\geq\|B_{a_n}\langle B_{a_n}, B_{a_n}\rangle^{-1}\langle f_n, B_{a_n}\rangle\|
\\&=|\langle f_n, B_{a_n}\rangle|
\\&=(1-|a_n|^2)^{\frac{m}{2}}|f_n(a_n)|.
\end{align*}
\end{lem}
\begin{proof}
Since $f_n$ and $B_n$ are both orthogonal to $B_1, B_2, \ldots, B_{n-1}$, we have
\begin{align}
&\|B_{a_1,\ldots,a_n}\langle B_{a_1,\ldots,a_n}, B_{a_1,\ldots,a_n}\rangle^{-1}\langle f, B_{a_1,\ldots,a_n}\rangle\|^2\nonumber
\\=&\|B_{a_1,\ldots,a_n}\langle B_{a_1,\ldots,a_n}, B_{a_1,\ldots,a_n}\rangle^{-1}\langle f_n, B_{a_1,\ldots,a_n}\rangle\|^2\nonumber
\\=&\|B_{a_1}\langle B_{a_1}, B_{a_1}\rangle^{-1}\langle f_n, B_{a_1}\rangle\|^2
+\|B_{a_1,a_2}\langle B_{a_1,a_2}, B_{a_1,a_2}\rangle^{-1}\langle f_n, B_{a_1,a_2}\rangle\|^2\nonumber
\\&+\ldots+
\|B_{a_1,\ldots,a_n}\langle B_{a_1,\ldots,a_n}, B_{a_1,\ldots,a_n}\rangle^{-1}\langle f_n, B_{a_1,\ldots,a_n}\rangle\|^2.\label{cyclic}
\end{align}
Note that for any $f\in\mathcal{H}^2(B^{m+1})$, the orthogonal projection of $f$ onto the space spanned by $B_1, B_2, \ldots, B_n$
is uniquely determined by $a_1,\ldots,a_n$, regardless of their orders. So
\begin{align*}
&B_{a_1}\langle B_{a_1}, B_{a_1}\rangle^{-1}\langle f_n, B_{a_1}\rangle
+B_{a_1,a_2}\langle B_{a_1,a_2}, B_{a_1,a_2}\rangle^{-1}\langle f_n, B_{a_1,a_2}\rangle
\\&+\ldots+B_{a_1,\ldots,a_n}\langle B_{a_1,\ldots,a_n}, B_{a_1,\ldots,a_n}\rangle^{-1}\langle f_n, B_{a_1,\ldots,a_n}\rangle
\\=&B_{a_n}\langle B_{a_n}, B_{a_n}\rangle^{-1}\langle f_n, B_{a_n}\rangle
+B_{a_n,a_1}\langle B_{a_n,a_1}, B_{a_n,a_1}\rangle^{-1}\langle f_n, B_{a_n,a_1}\rangle
\\&+\ldots+B_{a_n,a_1,\ldots,a_{n-1}}\langle B_{a_n,a_1,\ldots,a_{n-1}}, B_{a_n,a_1,\ldots,a_{n-1}}\rangle^{-1}
\langle f_n, B_{a_n,a_1,\ldots,a_{n-1}}\rangle,
\end{align*}
and \eqref{cyclic} equals
\begin{align*}
&\|B_{a_n}\langle B_{a_n}, B_{a_n}\rangle^{-1}\langle f_n, B_{a_n}\rangle\|^2
+\|B_{a_n,a_1}\langle B_{a_n,a_1}, B_{a_n,a_1}\rangle^{-1}\langle f_n, B_{a_n,a_1}\rangle\|^2
\\&+\ldots+\|B_{a_n,a_1,\ldots,a_{n-1}}\langle B_{a_n,a_1,\ldots,a_{n-1}}, B_{a_n,a_1,\ldots,a_{n-1}}\rangle^{-1}
\langle f_n, B_{a_n,a_1,\ldots,a_{n-1}}\rangle\|^2
\\ \geq&\|B_{a_n}\langle B_{a_n}, B_{a_n}\rangle^{-1}\langle f_n, B_{a_n}\rangle\|^2.
\end{align*}
\end{proof}

\begin{theo}
Subject to the maximum selection principle \eqref{MSP} we have
\begin{align}\label{converge}
\left\|\sum_{n=1}^N B_n\langle B_n, B_n\rangle^{-1}\langle f, B_n\rangle-f\right\|\rightarrow 0\quad (N\rightarrow\infty).
\end{align}
\end{theo}
\begin{proof}
From Bessel's inequality we have
$$\sum_{n=1}^\infty\left\|B_n\langle B_n, B_n\rangle^{-1}\langle f, B_n\rangle\right\|^2\leq\|f\|^2,$$
which implies that there exists a function $g\in\mathcal{H}^2(B^{m+1})$ such that
$$\sum_{n=1}^\infty B_n\langle B_n, B_n\rangle^{-1}\langle f, B_n\rangle=g$$
holds in the sense of $\mathcal{H}^2(B^{m+1})$. If \eqref{converge} is not true, then
$$h:=f-g\neq0,$$
so there exists a point $a\in B^{m+1}\setminus\bigcup_{i=1}^\infty\{a_i\}$
such that
$$\|B_{a}\langle B_{a}, B_{a}\rangle^{-1}\langle h, B_{a}\rangle\|=
|\langle h, B_{a}\rangle|=(1-|a|^2)^{\frac{m}{2}}|h(a)|=\delta>0.$$
Let
$$f_N=f-\sum_{n=1}^{N-1}B_n\langle B_n, B_n\rangle^{-1}\langle f, B_n\rangle,
\quad r_N=-\sum_{n=N}^\infty B_n\langle B_n, B_n\rangle^{-1}\langle f, B_n\rangle.$$
When $N$ is large enough,
\begin{align}
|\langle r_N, B_a\rangle|&=\|B_{a}\langle B_{a}, B_{a}\rangle^{-1}\langle r_N, B_{a}\rangle\|\nonumber
\\&\leq\|r_N\|=\left(\sum_{n=N}^\infty\|B_n\langle B_n, B_n\rangle^{-1}\langle f, B_n\rangle\|^2\right)^{1/2}<\delta/2.\label{ineq2}
\end{align}
So
$$|\langle f_N, B_a\rangle|=|\langle h-r_N, B_a\rangle|\geq
|\langle h, B_a\rangle|-|\langle r_N, B_a\rangle|>\delta/2.$$
By Lemma \ref{ineq1} we get
\begin{align*}
&\|B_{a_1,\ldots,a_{N-1},a}\langle B_{a_1,\ldots,a_{N-1},a}, B_{a_1,\ldots,a_{N-1},a}\rangle^{-1}\langle f, B_{a_1,\ldots,a_{N-1},a}\rangle\|
\\ \geq&\|B_{a}\langle B_{a}, B_{a}\rangle^{-1}\langle f_N, B_{a}\rangle\|=|\langle f_N, B_{a}\rangle|>\delta/2.
\end{align*}
On the other hand, from \eqref{ineq2} we know that
\begin{align*}
&\|B_{a_1,\ldots,a_{N-1},a_N}\langle B_{a_1,\ldots,a_{N-1},a_N}, B_{a_1,\ldots,a_{N-1},a_N}\rangle^{-1}\langle f, B_{a_1,\ldots,a_{N-1},a_N}\rangle\|
\\=&\|B_N\langle B_N, B_N\rangle^{-1}\langle f, B_N\rangle\|<\delta/2.
\end{align*}
Therefore we arrive at
\begin{align*}
&\|B_{a_1,\ldots,a_{N-1},a_N}\langle B_{a_1,\ldots,a_{N-1},a_N}, B_{a_1,\ldots,a_{N-1},a_N}\rangle^{-1}\langle f, B_{a_1,\ldots,a_{N-1},a_N}\rangle\|
\\<&\|B_{a_1,\ldots,a_{N-1},a}\langle B_{a_1,\ldots,a_{N-1},a}, B_{a_1,\ldots,a_{N-1},a}\rangle^{-1}\langle f, B_{a_1,\ldots,a_{N-1},a}\rangle\|,
\end{align*}
which contradicts with the maximum selection principle that we should not have chosen $a_N$ at the $N$-th step.
\end{proof}

Next we consider a convergence rate for adaptive Clifford TM system approximation. To deal with this, as in \cite{DT} we introduce a subclass of $\mathcal{H}^2(B^{m+1})$:
$$\mathcal{H}^2(B^{m+1}, M):=\left\{f\in\mathcal{H}^2(B^{m+1}): f=\sum_{k=1}^\infty B_{b_k}c_k~\mbox{with}~
\sum_{k=1}^\infty|c_k|\leq M<\infty\right\}.$$
We also need the following lemma.
\begin{lem}[\cite{DT}]\label{DT}
Let $\{d_n\}_{n=l}^\infty$ be a sequence of non-negative numbers
satisfying the inequalities
$$d_1\leq A,\quad d_{n+1}\leq d_n(1-d_n/A),\quad n=1,2,\ldots.$$
Then we have for each $n$
$$d_n\leq A/n.$$
\end{lem}
Now we can prove a convergence rate result.
\begin{theo}
If $f\in\mathcal{H}^2(B^{m+1}, M)$, then
$$\|f_N\|\leq \frac{2^{\frac{m}{2}}M}{\sqrt{N}},$$
where $f_N$ is the residue produced from the adaptive TM system approximation of $f$ at the $N$-th step.
\end{theo}
\begin{proof}
First, by Proposition \ref{IPS2} we have
$$\|f_1\|=\|f\|\leq\sum_{k=1}^\infty 2^{\frac{m}{2}}|c_k|\cdot\|B_{b_k}\|=\sum_{k=1}^\infty 2^{\frac{m}{2}}|c_k|\leq 2^{\frac{m}{2}}M,$$
and
\begin{align}
\|f_N\|^2&=|{\rm Sc}\langle f_N, f_N\rangle|\nonumber
\\&=|{\rm Sc}\langle f_N, f\rangle|\nonumber
\\&=\left|{\rm Sc}\left\langle f_N, \sum_{k=1}^\infty B_{b_k}c_k\right\rangle\right|\nonumber
\\&\leq\left|\left\langle f_N, \sum_{k=1}^\infty B_{b_k}c_k\right\rangle\right|\nonumber
\\&\leq 2^{\frac{m}{2}}M\sup_{k\geq 1}|\langle f_N, B_{b_k}\rangle|\nonumber
\\&\leq 2^{\frac{m}{2}}M\sup_{a\in B^{m+1}}|\langle f_N, B_a\rangle|\label{ineq3}.
\end{align}
Secondly, by Lemma \ref{ineq1} we get
\begin{align}
&\|B_N\langle B_N, B_N\rangle^{-1}\langle f, B_N\rangle\|\nonumber
\\=&\sup_{a\in B^{m+1}}\|B_{a_1,\ldots,a_{N-1},a}\langle B_{a_1,\ldots,a_{N-1},a}, B_{a_1,\ldots,a_{N-1},a}\rangle^{-1}
\langle f, B_{a_1,\ldots,a_{N-1},a}\rangle\|\nonumber
\\ \geq&\sup_{a\in B^{m+1}}|\langle f_N, B_{a}\rangle|\label{ineq4}.
\end{align}
So, from \eqref{ineq3} and \eqref{ineq4} we obtain
\begin{align*}
\|f_{N+1}\|^2&=\|f_N-B_N\langle B_N, B_N\rangle^{-1}\langle f, B_N\rangle\|^2
\\&=\|f_N\|^2-\|B_N\langle B_N, B_N\rangle^{-1}\langle f, B_N\rangle\|^2
\\&\leq\|f_N\|^2\left(1-\frac{\|f_N\|^2}{2^m M^2}\right).
\end{align*}
By Lemma \ref{DT} we conclude the proof.
\end{proof}

\begin{rem}\normalfont
Let $f\in L^2(S^m)$ (square integrable on $S^m$), where $f$ in not necessarily monogenic.
To get the adaptive approximation of $f$, without loss of generality we assume that $f$ is real-valued, and take
$$F(x):=T(f)(x):=\int_{\omega\in S^m}S(x,\omega)f(\omega)dS,\quad |x|<1,$$
where
$$S(x,\omega)=P(x,\omega)+Q(x,\omega)$$
is the monogenic Schwarz kernel,
$$P(x,\omega)=\frac{1}{\omega_m}\frac{1-|x|^2}{|x-\omega|^{m+1}}$$ is
the Poisson kernel and
\begin{align*}
Q(x,\omega)&={\rm NSc}\left(\int_0^1t^{m-1}(\overline{D}P)(tx,\omega)xdt\right)
\\&=\left(\frac{1}{\omega_m}\int_0^1\frac{(m+1)t^{m-1}(1-t^2|x|^2)}
{|tx-\omega|^{m+3}}dt\right){\rm NSc}(\overline{\omega}x)
\end{align*}
is the Cauchy-type harmonic conjugate of $P(x,\omega)$ on the unit
sphere $S^m$, which can be computed out explicitly with an expression in
elementary functions. As a consequence of boundedness of Hilbert transform on the sphere
(\cite{QY}), $T$ is a bounded operator from $L^2(S^m)$ to
$\mathcal{H}^2(B^{m+1})$. So $F\in\mathcal{H}^2(B^{m+1})$. The adaptive approximation of $f$ can be obtained by
the adaptive TM system approximation of $F$ through the relation
$$\lim_{r\rightarrow1^-}{\rm Sc}(F(r\xi))=f(\xi)$$
for a.e. $\xi\in S^m$.
\end{rem}

\begin{rem}\normalfont
The above theory can be similarly formulated in the context of the half space $\mathbb{R}^{m+1}_+$.
While for a real-valued function $f\in L^2(\mathbb{R}^m)$ we consider the Cauchy integral of
$f$:
$$F(x)=C(f):=\frac{-1}{\omega_m}\int_{\mathbb{R}^m}
\frac{\overline{\underline{y}-x}}{|\underline{y}-x|^{m+1}}
f(\underline{y})d\underline{y},\quad x\in \mathbb{R}^{m+1}_+,$$
where $\underline{y}=y_1e_1+\ldots+y_me_m$, $d\underline{y}=dy_1\cdots dy_m$.
We have
$F\in\mathcal{H}^2(\mathbb{R}^{m+1}_{+})$, and by
Sokhotsky--Plemelj formula we get
$$\lim_{x_0\rightarrow 0^{+}}F(x_0+\underline{x})=\frac{1}{2}f(\underline{x})
+\frac{1}{2}H(f)(\underline{x}),$$
where $H(f)=\sum_{i=1}^me_iR_i(f),$ and
$$R_i(f)(\underline{x}):=\frac{2}{\omega_m}\mbox{p.v.} \int_{\mathbb{R}^m}
\frac{y_i-x_i}{|\underline{y}-\underline{x}|^{m+1}}
f(\underline{y})d\underline{y}$$ is the $i$-th ($1\leq i\leq m$)
Riesz transform of $f$. The adaptive approximation of $f$ is then obtained by
the adaptive TM system approximation of $F$ through
$$2\lim_{x_0\rightarrow 0^{+}}\mbox{Sc}(F(x_0+\underline{x}))=f(\underline{x})$$
for a.e. $\underline{x}\in\mathbb{R}^m$.
\end{rem}

\vskip 0.8cm \noindent{\Large\textbf{Acknowledgements}}

\vskip 0.3cm \noindent
Jinxun Wang was supported by the National Natural Science Foundation of China (No. 11701105). Tao Qian was supported by The Science and Technology Development Fund, Macau SAR (No. 0123/2018/A3).

\end{document}